\documentclass[12pt,reqno]{amsart}
\usepackage{color,mathrsfs,epsfig}
\usepackage{graphicx}
\usepackage{amssymb}
\usepackage{color}

\newenvironment{proofof}[1]{\bigskip\noindent{\it Proof of~#1.}\rm}{\hfill $\Box$}

\newtheorem{theorem}{Theorem}[section]
\newtheorem{lemma}[theorem]{Lemma}

\theoremstyle{definition}

\theoremstyle{remark}

\numberwithin{equation}{section}

\begin{document}
\bibliographystyle{plain}
\title[An Improvement to a Berezin-Li-Yau type inequality]{An Improvement to a Berezin-Li-Yau type inequality for the Klein-Gordon Operator}
\author{Selma Yildirim Yolcu}
\address{Georgia Institute of Technology}
\email{selma@math.gatech.edu}
%\MSC Primary 35P15 \sep Secondary 35S99.

\begin{abstract}
In this article we improve a lower bound for $\sum_{j=1}^k\beta_j$
(a Berezin-Li-Yau type inequality) in \cite{HarYil}. Here
$\beta_j$ denotes the $j$th
eigenvalue of the Klein Gordon Hamiltonian  $H_{0,\Omega}=|p|$ when restricted to a bounded set
$\Omega\subset {\mathbb R}^n$. $H_{0,\Omega}$ can also be described as the generator of the
Cauchy stochastic process with a killing condition on $\partial
\Omega$. (cf. \cite{BanKul}, \cite{BanKul2}.) To do this, we adapt the proof of Melas (\cite{Melas}),  who
improved the estimate for the bound of $\sum_{j=1}^k\lambda_j$,   where $\lambda_j$ denotes the $j$th eigenvalue of
the Dirichlet Laplacian on a bounded domain in ${\mathbb R}^d$.
\end{abstract}

\maketitle

\section{Introduction}
In this article, we consider the pseudodifferential operator
$H_{0,\Omega}:=\sqrt{-\Delta}$ restricted
to an open bounded set $\Omega$ in ${\mathbb R}^d$. This operator is sometimes called the fractional Laplacian
with power $\frac{1}{2}$. ( cf. \cite{BanKul} and \cite{BanKul2}).
We note that  $H_{0,\Omega}$ is the generator of the Cauchy
stochastic process with a killing condition on $\partial\Omega$(cf. \cite{BanKul}, \cite{BanKul2}.)
Let $\beta_k$ denote the $k$th eigenvalue of $H_{0,\Omega}$ and $u_k$  denote
the corresponding normalized eigenfunction.
% give condition that guarantees the existence of the discrete spectrum
Then the eigenvalues $\beta_j$ satisfy
$$0<\beta_1<\beta_2\leq\beta_3\leq\cdots\leq\beta_j\leq
\cdots\;\to\;\infty,$$
where each eigenvalue is repeated according to its multiplicity. Throughout this article $|\Omega|$ denotes the volume of the set $\Omega$.

To show the analogy between the Dirichlet Laplacian and $H_{0,\Omega}$, we first mention similar results for the Dirichlet Laplacian. Let $\lambda_j$ be the eigenvalues of  the Dirichlet Laplacian on  $\Omega$. One such result is the so called Li-Yau inequality proved by P. Li and S.-T. Yau. In \cite{LiYau}, they proved that
\begin{equation}\label{LiYau}
\sum_{j=1}^{k}\lambda_j\geq \frac{dC_d}{d+2}|\Omega|^{-2/d}k^{1+2/d},
\end{equation}
where $C_d=4\pi\Gamma(1+d/2)^{2/d}$.

As mentioned in \cite{LapWei}, \eqref{LiYau}
can be obtained by a Legendre transform of an earlier result by Berezin\cite{Bere}.
Hence, instead of calling Li-Yau inequality, we prefer Berezin-Li-Yau inequality.

A.D. Melas improved the bound in the Berezin-Li-Yau inequality \eqref{LiYau} in \cite{Melas} and
proved that
\begin{equation}\label{melas}
\sum_{j=1}^{k}\lambda_j\geq \frac{dC_d}{d+2}|\Omega|^{-2/d}k^{1+2/d}+M_dk\frac{|\Omega|}{I(\Omega)},
\end{equation}
where the constant $M_d$ depends only on the dimension. Here $I(\Omega)$ is the moment of inertia,
which is defined as $\displaystyle{I(\Omega)=\min_{u\in{\mathbb R}^d}\int_{\Omega}|x-u|^2dx}$.

The improvement of the last inequality \eqref{melas} has recently been studied by many authors,( cf. \cite{KovVugWei}, \cite{Weidl}). More precisely, in \cite{KovVugWei}, H. Kova\v{r}\'{\i}k, S. Vugalter
and T. Weidl improved \eqref{melas} when $d=2$ and assuming geometric properties of the boundary of $\Omega$. Their proof is ingenious but somewhat intricate and they first state and prove their result in the case of polygons, then in the case of general domains. Moreover, their result has
a second term that has the order of $k$ as in the asymptotic behavior of the sum on the left
hand side of \eqref{LiYau}:
\begin{equation}\label{asymLiYau}
\sum_{j=1}^{k}\lambda_j=\frac{dC_d}{d+2}|\Omega|^{-2/d}k^{1+2/d}
+\tilde{C}_d\frac{|\partial\Omega|}{|\Omega|^{1+1/d}}k^{1+1/d}+o(k^{1+1/d})\;\;\mbox{as}\;k\to\infty.
\end{equation}

As stated in \cite{KovVugWei}, the correction term in \eqref{melas} is of larger
order than $k$, which appear in the asymptotics of \eqref{LiYau}.

Let's define the Riesz mean of order $\sigma$ as
$$R_{\sigma}(z)=\sum_{j}(z-\lambda_j)^{\sigma}_{+}.$$

Another analogous result is given in \cite{Weidl}, where T. Weidl found
a Berezin type bound for the Riesz mean $R_{\sigma}(z)$ when $\sigma>3/2$. The
second term in this bound is similar to the second term
in the asymptotics of $R_{\sigma}(z)$, up to a constant. His method utilizes sharp Lieb-Thirring inequalities for operator valued potentials.

A natural question is how this approach can be adapted to the case of Klein-Gordon operators.
This article answers this question and  improves the Berezin-Li-Yau type bound in \cite{HarYil}.
We follow the basic strategy of \cite{Melas}, with some important differences of detail. \\
We first state the analogue of the Weyl asymptotic formula and
the Berezin-Li-Yau type inequality in the case of Klein Gordon operators $H_{0,\Omega}$. In
\cite{HarYil}, E. Harrell and the author proved the following asymptotic formula:
\begin{theorem} (Analogue of the Weyl asymptotic formula) As $k\to\infty$,
\begin{equation}\label{weylasym}
\beta_k\sim \tilde{C_d}|\Omega|^{-1/d}k^{1/d},
\end{equation}
where $\tilde{C_d}=\sqrt{4\pi}\,\Gamma(1+d/2)^{1/d}$.
\end{theorem}
This theorem can be proved by adapting a proof of the Weyl asymptotic formula for
the Laplacian.

The analogue of the Berezin-Lie-Yau inequality  shown in \cite{HarYil} reads:
\begin{theorem} (Analogue of the Berezin-Lie-Yau inequality) The eigenvalues
$\beta_k$ of $H_{0,\Omega}$ satisfy
\begin{equation}\label{BerLiYau}
\sum_{j=1}^k\beta_k
\ge \frac{d\tilde{C_d}}{d+1}|\Omega|^{-1/d}k^{1+1/d}.
\end{equation}
\end{theorem}
As in the original Li-Yau paper \cite{LiYau}, the main tool used in the proof of
this theorem is a generalization of the lemma  which is attributed to H\"{o}rmander in \cite{LiYau}. This result is also sharp in the
sense of the Weyl asymptotic formula as in the case of the Laplacian.
\section{Statement and Proof of the Theorem}
The main result of this paper is given below:
\begin{theorem}\label{corrected} For $k\geq 1$ and the bounded set $\Omega$,
\begin{equation}\label{thmresult}
\sum_{j=1}^k\beta_j
\geq \frac{d\tilde{C_d}}{d+1}|\Omega|^{-1/d}k^{1+1/d}
+\tilde{M_d} \frac{|\Omega|^{1+1/d}}{I(\Omega)} k^{1-1/d},
\end{equation}
where $\tilde{C_d}=\sqrt{4\pi}\,\Gamma(1+d/2)^{1/d}$ and the constant $\displaystyle{\tilde{M_d}}$ depends
only on the dimension $d$.
\end{theorem}
Observe that, in \eqref{melas}, the power of $k$ in the first term is $1+2/d$ while in \eqref{thmresult} the corresponding power is $1+1/d$. This is not surprising because the Klein-Gordon operator can be viewed as the square root of Laplacian in ${\mathbb R}^d$. Also, the improvement in \eqref{melas} consists of ${|\Omega|}/{I(\Omega)}$ and in \eqref{thmresult} we have ${|\Omega|^{1+1/d}}/{I(\Omega)}$. Moreover, the difference between the powers of the k terms on the right hand side of \eqref{thmresult} is $2/d$ as in \eqref{melas}.

First, we will state and prove the following lemma, which is the crucial step in proving the theorem.
%%%%%%%%%%%%%%%%%%%%%%%%%%%%%%%%%%%%%%%%%%%%%%%%%%%%%%%%%%%%%%%%%%%%%%%%%%
\begin{lemma}\label{lemma}
Let $d\geq 2$ and $\varphi:[0,\infty)\to [0,\infty)$ be a decreasing, absolutely
continuous function. Assume that
\begin{equation}\label{mp0}0\leq-\varphi'(x)\leq m,\;\;x>0.\end{equation}
Then,
\begin{eqnarray}\int_0^{\infty}x^d\varphi(x)dx&\geq& \frac{1}{d+1}\left(d\int_{0}^{\infty}x^{d-1}\varphi(x)dx\right)^{1+1/d}\varphi(0)^{-1/d}\nonumber\\
&+&\frac{\varphi(0)^{2+1/d}}{6m^2(d^2-1)}
\left(d\int_{0}^{\infty}x^{d-1}\varphi(x)dx\right)^{1-1/d}.\label{xnp}
\end{eqnarray}
\end{lemma}
\begin{proof}
Let us first define \begin{equation}\label{trk}\eta(x)=\frac{1}{\varphi(0)}\,\varphi\left(\frac{\varphi(0)}{m}\,x\right).\end{equation}
Then $\eta(0)=1$ and $0\leq-\eta'(x)\leq 1$. To ease the notation, define $f(x):=-\eta'(x)$ for $x\geq 0$.
Hence, $0<f(x)<1$ for $x>0$ and  $\displaystyle{\int_{0}^{\infty}}f(x)dx=\eta(0)=1$. Now,
define
\begin{equation}\label{xn-1p}A:=\int_0^{\infty}x^{d-1}\eta(x)dx
\qquad\mbox{and}\qquad B:=\int_{0}^{\infty}x^d\eta(x)dx.\end{equation}
Assume that $B<+\infty$, as otherwise the result is immediate. Thus, we can find
a sequence $\{R_j\}$ such that $R_j\to\infty$
and $\displaystyle{R_j^{d+1}\eta(R_j)\to 0}$ as $j\to \infty$.
Then, using integration by parts we get
$$\int_0^{\infty}x^{d}f(x)dx=Ad,\qquad\mbox{and}\qquad\int_0^{\infty}x^{d+1}f(x)dx\leq(d+1)B.$$
By the initial value theorem, there exist an $\alpha\geq 0$ such that
\begin{equation}\label{alphaA}
\int_{\alpha}^{\alpha+1}x^{d-1} dx=(Ad)^{1-1/d}\end{equation}
and
\begin{equation}\label{alphaB}
\int_{\alpha}^{\alpha+1}x^{d+1}dx\leq
\int_{0}^{\infty}x^{d+1}f(x)dx\leq (d+1)B.\end{equation}
As we shall see later, the key point in the proof of the lemma is the inequality \begin{equation}\label{polyxy}(d-1)x^{d+1}-(d+1)y^2x^{d-1}+2y^{d+1}
\geq 2y^{d-1}(x-y)^2\end{equation}
for $y>0$ and $x\geq 0$.
The proof of \eqref{polyxy} is straightforward. Indeed, first divide both sides by $y^{d+1}$. Then, by setting $\tau=\dfrac{x}{y}$ we get the polynomial
$$g(\tau):=(d-1)\tau^{d+1}-(d+1)\tau^{d-1}-2\tau^2+4\tau
=(\tau-1)^2\tau\left(\sum_{k=0}^{d-3}(2k+4)\tau^k+(d-1)\tau^{d-2}\right).$$
An induction on $d$ leads to $g(\tau)\geq 0$. Now, integrating \eqref{polyxy}
from $\alpha$ to $\alpha+1$ and using \eqref{alphaA} and \eqref{alphaB} we get
\begin{eqnarray*}
(d+1)(d-1)B-(d+1)y^2(Ad)^{1-1/d}+2y^{d+1}
&\geq& 2y^{d-1}\int_{\alpha}^{\alpha+1}(x-y)^2dx\\
&\geq& 2y^{d-1}\int_{-1/2}^{1/2}s^2ds\\
&=& \frac{y^{d-1}}{6}.
\end{eqnarray*}
Choosing $y=(Ad)^{1/d}$ yields
$$B\geq \frac{1}{d+1}(Ad)^{1+1/d}+\frac{1}{6(d^2-1)}(Ad)^{1-1/d},$$
or, equivalently,
\begin{eqnarray*}\int_{0}^{\infty}x^d\eta(x)dx &\geq& \frac{1}{d+1}\left(d\int_0^{\infty}x^{d-1}\eta(x)dx\right)^{1+1/d}
+\frac{1}{6(d^2-1)}\left(d\int_0^{\infty}x^{d-1}\eta(x)dx\right)^{1-1/d},\end{eqnarray*}
which together with \eqref{trk} gives
\begin{eqnarray}\int_0^{\infty}x^d\varphi(x)dx &\geq& \frac{1}{d+1}\left(d\int_{0}^{\infty}x^{d-1}\varphi(x)dx\right)^{1+1/d}\varphi(0)^{-1/d}\nonumber\\
&+&\frac{\varphi(0)^{2+1/d}}{6m^2(d^2-1)}
\left(d\int_{0}^{\infty}x^{d-1}\varphi(x)dx\right)^{1-1/d},\label{xnpresult}
\end{eqnarray}
concluding the proof.\end{proof}
Let us now prove the theorem by using the lemma.

\begin{proofof}{Theorem \ref{corrected}} Let the Fourier transform of
each eigenfunction $u_j$ corresponding to the $j$th eigenvalue $\beta_j$ be denoted by
$$\hat{u}_j(\xi)=\frac{1}{(2\pi)^{d/2}}\int_{\Omega}e^{-ix\cdot\xi}u_j(x)dx.$$
Since the set of eigenfunctions $\{u_j\}_{j=1}^{\infty}$ forms an orthonormal
set, the set of $\{\hat{u}_j(\xi)\}_{j=1}^{\infty}$ is also an orthonormal set
in ${\mathbb R}^d$ by using the Plancherel's theorem.
Set
$$F(\xi):=\sum_{j=1}^k|\hat{u}_j(\xi)|^2.$$
Now we will use the decreasing radial rearrangement of $F(\xi)$ and the coarea formula
to get the condition in the lemma.
Let $F^*(\xi)=\varphi(|\xi|)$ be the decreasing radial rearrangement of
$F$. We may assume that $\varphi$ is absolutely continuous.
Let
$\mu(t)=|\{F^*(\xi)>t\}|=|\{F(\xi)>t\}|.$ Then, $\mu(\varphi(x))=\omega_dx^d$. By the coarea formula,
$$\mu(t)=\int_{t}^{|\Omega|/(2\pi)^{d}}\int_{\{F=x\}}|\nabla F|^{-1}d\sigma_xdx.$$
Then,
\begin{equation}\label{muvar}
-\mu'(\varphi(x))=\int_{\{F=\varphi(x)\}}|\nabla F|^{-1}d\sigma_{\varphi(x)}.
\end{equation}
Next we will estimate $|\nabla F|$:
$$\sum_{j=1}^{k}|\nabla \hat{u}_j(\xi)|^2\leq\frac{1}{(2\pi)^{d}}\int_{\Omega}|ixe^{-ix\cdot\xi}|^2dx
=\frac{I(\Omega)}{(2\pi)^{d}},$$
where $I(\Omega)$, the moment
of inertia, is defined as follows:
$$I(\Omega)=\min_{u\in{\mathbb R}^d}\int_{\Omega} |x-u|^2dx.$$
After translation, we may assume that
$$I(\Omega)=\int_{\Omega} |x|^2dx.$$
Observe that for every $\xi$,
\begin{equation}\label{thm:gradFbound}|\nabla F(\xi)|\leq 2\left(\sum_{j=1}^{k}|\hat{u}_j(\xi)|^2\right)^{1/2}
\left(\sum_{j=1}^{k}|\nabla \hat{u}_j(\xi)|^2\right)^{1/2}\leq 2(2\pi)^{-d}\sqrt{|\Omega|I(\Omega)}.\end{equation}
By letting $m:=2(2\pi)^{-d}\sqrt{|\Omega|I(\Omega)}$ and using \eqref{thm:gradFbound} in
\eqref{muvar}, we obtain
\begin{eqnarray*}
-\mu'(\varphi(x))&\geq& m^{-1}{\rm Vol}_{n-1}(\{F=\varphi(x)\})\\
&\geq& m^{-1}d\omega_d x^{d-1}.
\end{eqnarray*}
On the other hand, differentiating $\mu(\varphi(x))$ yields
$\mu'(\varphi(x))\varphi'(x)=d\omega_dx^{d-1}$. Thus,
\begin{equation}\label{condinlemma}
0\leq-\varphi'(x)\leq m,\end{equation}
which is the required condition in the lemma.
Thus, it remains to prove the theorem by using the lemma. Observe that
\begin{equation}\label{Fintk}
\int_{{\mathbb R}^d}F(\xi)d\xi=k.
\end{equation}
Observe that because the $u_j$'s form an orthonormal set in $L^2(\Omega)$, by Bessel's inequality
\begin{equation}
0\leq F(\xi)\leq\frac{|\Omega|}{(2\pi)^{d}}.\label{fksiless}
\end{equation}
Since
$$\beta_j=\langle u_j,H_{0,\Omega}u_j\rangle=\int_{{\mathbb R}^d}|\xi||\hat{u}_j(\xi)|^2d\xi,$$
with the definition of $F$, we have
\begin{equation}\label{Fksiintk}
\int_{{\mathbb R}^d}|\xi|F(\xi)d\xi=\sum_{j=1}^k\beta_j.
\end{equation}
Hence,
\begin{equation}\label{eqnkA}
k=\int_{{\mathbb R}^d} F(\xi)d\xi=\int_{{\mathbb R}^d} F^*(\xi)d\xi=d\omega_d\int_{0}^{\infty}x^{d-1}\varphi(x)dx,\end{equation}
and
\begin{equation}\label{sumxiB}\sum_{j=1}^k\beta_j=\int_{{\mathbb R}^d} |\xi|F(\xi)d\xi=\int_{{\mathbb R}^d} |\xi|F^*(\xi)d\xi=d\omega_d\int_{0}^{\infty}x^{d}\varphi(x)dx,\end{equation}
where $\omega_d$ denotes the volume of the d-dimensional unit ball. The equations \eqref{eqnkA}, \eqref{sumxiB}, when combined with Lemma \ref{lemma} yield
\begin{equation}\label{inwithvarphi}
\sum_{j=1}^{k}\beta_j\geq \frac{d}{d+1}{\omega_d}^{-1/d}
\varphi(0)^{-1/d}k^{1+1/d}
+\frac{d}{6m^2(d^2-1)}\omega_d^{1/d}\varphi(0)^{2+1/d}k^{1-1/d}.
\end{equation}
Define
$$h(t)=\frac{d}{d+1}{\omega_d}^{-1/d}k^{1+1/d}
t^{-1/d}
+\frac{Cd}{m^2(d^2-1)}\omega_d^{1/d}k^{1-1/d}t^{2+1/d},$$
where $C$ is a constant to be chosen later. Observe that the function $h$ is decreasing on
$$0<t\leq \left(\frac{m^2(d-1)k^{2/d}}{C(2d+1)\omega_d^{2/d}}\right)^{d/(d+2)}.$$
Let $R$ be the number such that $|\Omega|=\omega_dR^d$. Then,
$$I(\Omega)\geq \int_{B(R)}|x|^2dx=\frac{d\omega_dR^{d+2}}{d+2}
,$$
where $B(R)$ is the ball of radius $R$.
Then,
$$m=2(2\pi)^{-d}\sqrt{|\Omega|I(\Omega)}
\geq2(2\pi)^{-d}\sqrt{\frac{d}{d+2}\omega_d^{-2/d}|\Omega|^{(2d+2)/d}}
\geq (2\pi)^{-d}\omega_d^{-1/d}|\Omega|^{(d+1)/d}.
$$
Choosing $C=\min\left\{\dfrac{1}{6},\dfrac{m^2(d-1)k^{2/d}(2\pi)^{d+2}}{(2d+1)\omega_d^{2/d}|\Omega|^{1+2/d}}
\right\}$
will guarantee that
$$\left(\frac{m^2(d-1)k^{2/d}}{C(2d+1)\omega_d^{2/d}}\right)^{d/(d+2)}\geq (2\pi)^{-d}|\Omega|.$$
Hence,
the function $h$ is decreasing on $\big(0,(2\pi)^{-d}|\Omega|\big]$. Since
$0<\varphi(0)\leq (2\pi)^{-d}|\Omega|,$ and $h$ is decreasing, we can replace
$\varphi(0)$ in \eqref{inwithvarphi} with $(2\pi)^{-d}|\Omega|$. Therefore,
\eqref{inwithvarphi} and the fact that
$\omega_d=\dfrac{\pi^{d/2}}{\Gamma\left(1+d/2\right)}$ result in the following inequality:
\begin{equation}\label{impbound}
\sum_{j=1}^k\beta_j
\geq \frac{\sqrt{4\pi}d}{d+1}
\left(\dfrac{\Gamma\left(1+d/2\right)}{|\Omega|}\right)^{1/d}k^{1+1/d}
+\frac{Cd}
{8\sqrt{\pi}(d^2-1)(\Gamma(1+d/2))^{1/d}}\frac{|\Omega|^{1+1/d}}{I(\Omega)}k^{1-1/d}.
\end{equation}
Let $\tilde{M_d}:=\dfrac{Cd}
{8\sqrt{\pi}(d^2-1)(\Gamma(1+d/2))^{1/d}}$. Then \eqref{impbound} can be written as
\begin{equation}\label{result}
\sum_{j=1}^k\beta_j
\geq \frac{d\tilde{C_d}}{d+1}|\Omega|^{-1/d}k^{1+1/d}
+\tilde{M_d} \frac{|\Omega|^{1+1/d}}{I(\Omega)} k^{1-1/d},
\end{equation}
where $\tilde{C_d}=\sqrt{4\pi}\Gamma(1+d/2)^{1/d}$. Recall that the first term on the right of \eqref{result} is same bound as in \cite{HarYil}.
\end{proofof}

\subsection*{Acknowledgements}
The author wishes to thank Evans Harrell, Lotfi Hermi, and T{\"u}rkay Yolcu for suggestions, comments and helpful discussions.

\end{document}